\documentclass{article}
\usepackage{amsfonts}
\usepackage{amsmath}
\usepackage{amssymb}
\usepackage{amsthm}
\usepackage[margin=1.3in]{geometry}
\usepackage{hyperref}
\usepackage{placeins}

\newtheorem{theorem}{Theorem}[section]
\theoremstyle{definition}
\newtheorem{example}[theorem]{Example}
\newtheorem{remark}[theorem]{Remark}

\makeatletter

\renewcommand*{\@fnsymbol}[1]{\@arabic{#1}}

\makeatother

\newtheorem{thm}{Theorem}

\newcommand{\red}{\text{red}}

\newcommand{\SG}{\mathcal{S}}

\newcommand{\cheader}[1]{\multicolumn{1}{|c|}{#1}}

\newcommand\gd{\delta}

\newcommand\gs{\sigma}

\newcommand\gS{\Sigma}

\newcommand\gt{\tau}

\newcommand\set[1]{\ensuremath{\{#1\}}}

\newcommand\bigpar[1]{\bigl(#1\bigr)}

\newcommand\Bigpar[1]{\Bigl(#1\Bigr)}

\newcommand\E{\operatorname{\mathbb E{}}}

\newcommand\PP{\operatorname{\mathbb P{}}}

\newcommand\Var{\operatorname{Var}}

\newcommand\Cov{\operatorname{Cov}}

\newcommand{\tend}{\longrightarrow}

\newcommand\dto{\overset{\mathrm{d}}{\tend}}

\newcommand\innprod[1]{\langle#1\rangle}

\newcommand\intoi{\int_0^1}

\newcommand\oi{[0,1]}

\newcommand\dd{\,\mathrm{d}}

\newcommand\xn[1]{X_{#1,n}}

\newcommand\fall[1]{^{\underline{#1}}}

\newcommand{\refT}[1]{Theorem~\ref{#1}}

\newcommand\marginal[1]{\marginpar{\raggedright\parindent=0pt\tiny #1}}

\newcommand\REM[1]{{\raggedright\texttt{[#1]}\par\marginal{XXX}}}

\newcommand\ett[1]{\boldsymbol1\left[#1\right]} 

\DeclareMathOperator*{\sumx}{\sum\nolimits^{*}}

\begin{document}

\title{On the Asymptotic Statistics of the Number of Occurrences of Multiple Permutation Patterns}

\author{Svante Janson\thanks{Department of Mathematics, Uppsala University, Uppsala, Sweden. [svante.janson@math.uu.se]}, \; Brian Nakamura\thanks{CCICADA/DIMACS, Rutgers University-New Brunswick, Piscataway, NJ, USA. [brian.nakamura@rutgers.edu]}, \; and Doron Zeilberger\thanks{Mathematics Department, Rutgers University-New Brunswick, Piscataway, NJ, USA. [zeilberg@math.rutgers.edu]}}


\maketitle


\begin{abstract}

	We study statistical properties of the random variables $X_{\sigma}(\pi)$, the number of occurrences of the pattern $\sigma$ in the permutation $\pi$. We present two contrasting approaches to this problem: traditional probability theory and the ``less traditional'' computational approach. Through the perspective of the first one, we prove that for any pair of patterns $\sigma$ and $\tau$, the random variables $X_{\sigma}$ and $X_{\tau}$ are jointly asymptotically normal (when the permutation is chosen from $S_{n}$). From the other perspective, we develop algorithms that can show asymptotic normality and joint asymptotic normality (up to a point) and derive explicit formulas for quite a few moments and mixed moments empirically, yet rigorously. The computational approach can also be extended to the case where permutations are drawn from a set of pattern avoiders to produce many empirical moments and mixed moments. This data suggests that some random variables are not asymptotically normal in this setting.

\end{abstract}


\section{Introduction}

The primary area of interest in this article is the study of patterns in
permutations. We will denote the set of length $n$ permutations by
$\SG_{n}$. Let $a_{1} a_{2} \ldots a_{k}$ be a sequence of $k$ distinct real
numbers. The \emph{reduction} of this sequence, which is denoted by
$\red(a_{1} \ldots a_{k})$, is the length $k$ permutation $\pi_{1} \ldots
\pi_{k} \in \SG_{k}$ such that order-relations are preserved (i.e., $\pi_{i}
< \pi_{j}$ if and only if $a_{i} < a_{j}$ for every $i$ and $j$). Given a
(permutation) pattern $\tau \in \SG_{k}$, we say that a permutation $\pi =
\pi_{1} \ldots \pi_{n} \in \SG_{n}$ \emph{contains} the pattern $\tau$ if
there exists $1 \leq i_{1} < i_{2} < \ldots < i_{k} \leq n$ such that
$\red(\pi_{i_{1}} \pi_{i_{2}} \ldots \pi_{i_{k}}) = \tau$. Each such
subsequence in $\pi$ will be called an \emph{occurrence} of the pattern
$\tau$. If $\pi$ contains no such subsequence, it is said to \emph{avoid}
the pattern $\tau$. Additionally, we will denote the number of occurrences
of the pattern $\tau$ in permutation $\pi$ by $N_{\tau}(\pi)$ (e.g., $\pi$
avoids the pattern $\tau$ if and only if $N_{\tau}(\pi) = 0$).

For any pattern $\tau$ and integer $n \geq 0$, we define the set
\begin{equation}
	\SG_{n}(\tau) := \{ \pi \in \SG_{n} \; : \; \pi \text{ avoids the pattern } \tau \}
\end{equation}
and also define $s_{n}(\tau) := | \SG_{n}(\tau) |$. The patterns $\sigma$ and $\tau$ are said to be \emph{Wilf-equivalent} if $s_{n}(\sigma) = s_{n}(\tau)$ for all $n \geq 0$. We may also consider the more general set
\begin{equation}
	\SG_{n}(\tau,r) := \{ \pi \in \SG_{n} \; : \; \pi \text{ contains exactly } r \text{ occurrences of } \tau \}.
\end{equation}
We will analogously define $s_{n}(\tau,r) := | \SG_{n}(\tau,r) |$.

A classical problem in this area is to find an enumeration for these sets or at the least, to study properties of the generating function encoding the enumerating sequence (for example, is it rational/algebraic/holonomic?). However, it is not even known if these generating functions are always holonomic. In general, the enumeration problem gets very difficult very quickly. Patterns up to length $3$ are well-understood, but there are basic unresolved questions even for length $4$ patterns. For example, it is known that there are three Wilf-equivalence classes for length $4$ patterns: $1234$, $1324$, and $1342$. While the enumeration problems have been solved for $1234$ and $1342$, no exact enumeration (or even asymptotics) is known for $1324$.

A (probabilistic) variation of this problem was posed by Joshua Cooper \cite{Cooper}: Given two (permutation) patterns $\sigma$ and $\tau$, what is the expected number of copies of $\sigma$ in a permutation chosen uniformly at random from $\SG_{n}(\tau)$? We note that if the enumeration of $\SG_{n}(\tau)$ is known, this question is equivalent to counting the total number of occurrences of $\sigma$ in permutations from $\SG_{n}(\tau)$, or put more precisely, to compute
\begin{equation}
	T_{n}(\sigma, \tau) := \mathop{\sum} \limits_{\pi \in \SG_{n}(\tau)} {N_{\sigma}(\pi)}.
\end{equation}

B{\'o}na first addressed the question for $\tau = 132$ when $\sigma$ is either the increasing or decreasing permutation in \cite{Bona2}. He shows how to derive the generating functions for $T_{n}(1 2 \ldots k, 132)$ and $T_{n}(k \ldots 2 1, 132)$, the total number of occurrences of $1 2 \ldots k$ in $\SG_{n}(\tau)$ and occurrences of $k \ldots 2 1$ in $\SG_{n}(\tau)$, respectively. In \cite{Bona4}, B{\'o}na also shows that $T_{n}(213, 132) = T_{n}(231, 132) = T_{n}(312, 132)$ for all $n$ and provides an explicit formula for them. Rudolph \cite{Rudolph} also proves some conditions on when two patterns, say $p$ and $q$, occur equally frequently in $\SG_{n}(132)$ (i.e., $T_{n}(p,132) = T_{n}(q,132)$ for all $n$).

In \cite{Homberger}, Homberger answers the analogous question when $\tau = 123$ and shows that there are three non-trivial cases to consider: $T_{n}(132, 123)$, $T_{n}(231,123)$, and $T_{n}(321,123)$. He finds generating functions and explicit formulas for each one.

We will consider a more general problem. Given the pattern $\tau$, suppose that a permutation $\pi$ is chosen uniformly at random from $\SG_{n}(\tau)$. Given another pattern $\sigma$, we define the random variable $X_{\sigma}(\pi) := N_{\sigma}(\pi)$, the number of copies of $\sigma$ in $\pi$. Observe that $T_{n}(\sigma, \tau) = \mathbb{E}[X_{\sigma}]$, the expected value of $X_{\sigma}$ (i.e., the first moment of the random variable). The focus of this paper is to study higher moments for $X_{\sigma}$ as well as mixed moments between two such random variables that count different patterns. We will consider the case where the permutation $\pi$ is randomly chosen from $\SG_{n}$ as well as some cases where $\pi$ is chosen from $\SG_{n}(\tau)$ (for various patterns $\tau$).

In this paper, we approach the problem from two different angles. On one end, we will present (human-derived) results proving that the random variables are jointly asymptotically normal when the permutations are chosen at random from $\SG_{n}$. Unfortunately, the techniques do not naturally extend to the scenario when the permutations are chosen from $\SG_{n}(\tau)$. On the other end, we present a computational approach that can quickly and easily compute many empirical moments for the general case (permutations chosen from $\SG_{n}(\tau)$). In addition, for the case where permutations are chosen from $\SG_{n}$, the computational approach can rigorously produce closed-form formulas for quite a few moments and mixed moments of the random variables.

This paper is organized as follows. In Section~\ref{SECfunceqn}, we review and outline the functional equations enumeration approach developed in \cite{BN-GWILF2, NZ-GWILF}. In Section~\ref{SECmoments}, we derive both rigorous results and empirical values for higher order moments and mixed moments for various random variables $X_{\sigma}$. In Section~\ref{SECasymom}, we show that the random variables are jointly asymptotically normal when the permutations are randomly chosen from $\SG_{n}$. In Section~\ref{SECconcl}, we conclude with some final remarks and observations.


\section{Enumerating with functional equations}\label{SECfunceqn}

For various patterns $\tau$, functional equations were derived for enumerating permutations with $r$ occurrences of $\tau$ in \cite{BN-GWILF2, NZ-GWILF, NoonZeil}. These functional equations were then used to derive enumeration algorithms. We briefly review the relevant results here. The curious reader can see \cite{BN-GWILF2, NZ-GWILF, NoonZeil} for more details.


\subsection{Functional equations for single patterns}

Given a (fixed) pattern $\tau$ and non-negative integer $n$, we define the polynomial:
\begin{equation}
	f_{n}(\tau; \; t) := \mathop{\sum} \limits_{\pi \in \SG_{n}} {t^{N_{\tau}(\pi)}} .
\end{equation}
Recall that the coefficient of $t^{r}$ is exactly $s_{n}(\tau, r)$. For certain patterns $\tau$, a multi-variate polynomial $P_{n}(\tau; \; t; \; x_{1}, \ldots, x_{n})$ was defined so that $P_{n}(\tau; \; t; \; 1, \ldots, 1) = f_{n}(\tau; \; t)$ and that functional equations could be derived for the $P_{n}$ polynomial.

The pattern $\tau = 123$ was considered in \cite{NZ-GWILF, NoonZeil}, and the polynomial $P_{n}$ was defined to be:
\begin{equation}
	P_{n}(123; \; t; \; x_{1}, \ldots, x_{n}) := \mathop{\sum} \limits_{\pi \in \SG_{n}} { \left( t^{N_{123}(\pi)} \mathop{\prod} \limits_{i=1}^{n} {x_{i}^{| \{ (a,b) \; : \; \pi_{a}=i<\pi_{b}, \; 1 \leq a < b \leq n \} |}}  \right)} .
\end{equation}
It was shown that this $P_{n}$ satisfies the functional equation:
\begin{thm}
	For the pattern $\tau = 123$,
	\begin{equation}
		P_{n}(123; \; t; \; x_{1}, \ldots, x_{n}) = \mathop{\sum} \limits_{i=1}^{n} {x_{i}^{n-i} \cdot P_{n-1}(123; \; t; \; x_{1}, \ldots, x_{i-1}, t x_{i+1}, \ldots, t x_{n})}. \tag{FE123} \label{FE123}
	\end{equation}
\end{thm}

\noindent Since $P_{1}(123; \; t; \; x_{1}) = 1$, the functional equation can be used to recursively compute our desired quantity $P_{n}(123; \; t; \; 1, \ldots, 1) = f_{n}(123; \; t)$.

Similarly, in \cite{BN-GWILF2}, the polynomial $P_{n}$ was defined for the pattern $\tau = 132$ so that it satisfied the functional equation:
\begin{thm}
	For the pattern $\tau = 132$,
	\begin{equation}
		P_{n}(132; \; t; \; x_{1}, \ldots, x_{n}) = \mathop{\sum} \limits_{i=1}^{n} {x_{1} x_{2} \ldots x_{i-1} \cdot P_{n-1}(132; \; t; \; x_{1}, \ldots, x_{i-1}, t x_{i+1}, \ldots, t x_{n})}. \tag{FE132} \label{FE132}
	\end{equation}
\end{thm}

\noindent Again $P_{1}(132; \; t; \; x_{1}) = 1$, so the functional equation can be used to recursively compute our desired quantity $P_{n}(132; \; t; \; 1, \ldots, 1) = f_{n}(132; \; t)$.

The same was also done for the pattern $\tau = 231$ in \cite{BN-GWILF2}. Although $f_{n}(231; \; t) = f_{n}(132; \; t)$, redeveloping the approach directly for the pattern $231$ allows us to consider the patterns $132$ and $231$ simultaneously. For $231$, the polynomial $P_{n}$ was defined so that it satisfies the functional equation:
\begin{thm}
	For the pattern $\tau = 231$,
	\begin{equation}
		P_{n}(231; \; t; \; x_{1}, \ldots, x_{n}) = \mathop{\sum} \limits_{i=1}^{n} {x_{1}^{0} x_{2}^{1} \ldots x_{i}^{i-1} \cdot P_{n-1}(231; \; t; \; x_{1}, \ldots, x_{i-1}, t x_{i} x_{i+1}, x_{i+2}, \ldots, x_{n})}. \tag{FE231} \label{FE231}
	\end{equation}
\end{thm}

\noindent We again have that $P_{1}(231; \; t; \; x_{1}) = 1$, so the functional equation can be used to recursively compute our desired quantity $P_{n}(231; \; t; \; 1, \ldots, 1) = f_{n}(231; \; t)$.

The approach for the pattern $123$ was also extended to the pattern $\tau = 1234$ in \cite{NZ-GWILF}. The polynomial $P_{n}(1234; \; t; \; x_{1}, \ldots, x_{n}; \; y_{1}, \ldots, y_{n})$ was defined so that $P_{n}(1234; \; t; \; 1 \text{ [n times]}; \; 1 \text{ [n times]}) = f_{n}(1234; \; t)$ and in such a way that it satisfies the functional equation:
\begin{thm}
	For the pattern $\tau = 1234$,
	\begin{gather}
		P_{n}(1234; \; t; \; x_{1}, \ldots, x_{n}; \; y_{1}, \ldots, y_{n}) =\notag \\ 
		\mathop{\sum} \limits_{i=1}^{n} {y_{i}^{n-i} \cdot P_{n-1}(1234; \; t; \; x_{1}, \ldots, x_{i-1}, t x_{i+1}, \ldots, t x_{n}; \; y_{1}, \ldots, y_{i-1}, x_{i} y_{i+1}, \ldots, x_{i} y_{n})}. \tag{FE1234} \label{FE1234}
	\end{gather}
\end{thm}

\noindent Since $P_{1}(1234; \; t; \; x_{1}; \; y_{1}) = 1$, the functional equation can be used to recursively compute our desired quantity $P_{n}(1234; \; t; \; 1 \text{ [n times]}; \; 1 \text{ [n times]}) = f_{n}(1234; \; t)$.


\subsection{Merging functional equations for multiple patterns}

It is also straight-forward to consider multiple patterns simultaneously if their corresponding functional equations are known, as shown in \cite{BN-GWILF2}. For example, suppose that we want to consider the two patterns $\sigma = 123$ and $\tau = 132$ simultaneously. We can extend the $f_{n}$ polynomial in the natural way to:
\begin{equation}
	f_{n}(\sigma, \tau; \; s, t) := \mathop{\sum} \limits_{\pi \in \SG_{n}} {s^{N_{\sigma}(\pi)} t^{N_{\tau}(\pi)} }.
\end{equation}

In \cite{BN-GWILF2}, the polynomial $P_{n}(123, 132; \; s, t; \; x_{1}, \ldots, x_{n}; \; y_{1}, \ldots, y_{n})$ was defined so that 
\begin{equation}
	P_{n}(123, 132; \; s, t; \; 1 \text{ [n times]}; \; 1 \text{ [n times]}) = f_{n}(123, 132; \; s, t).
\end{equation}
The following functional equation was then derived:
\begin{thm} \label{THM123n132}
	For the patterns $\sigma = 123$ and $\tau = 132$,
	\begin{gather*}
		P_{n}(123, 132; \; s, t; \; x_{1}, \ldots, x_{n}; \; y_{1}, \ldots, y_{n}) = \\ 
		\mathop{\sum} \limits_{i=1}^{n} {x_{i}^{n-i} \cdot y_{1} y_{2} \ldots y_{i-1} \cdot P_{n-1}(123, 132; \; s, t; \; x_{1}, \ldots, x_{i-1}, s x_{i+1}, \ldots, s x_{n}; \; y_{1}, \ldots, y_{i-1}, t y_{i+1}, \ldots, t y_{n})}. 
	\end{gather*}
\end{thm}

\noindent Observe that we combined the functional equations for the individual patterns $123$ and $132$ by re-labeling the $x_{i}$ variables for $132$ to $y_{i}$, merging the reductions in the $P_{n-1}$ in the natural way, and multiplying the coefficient terms for the $P_{n-1}$ within the summands. We again have that $P_{1}(123, 132; \; s, t; \; x_{1}; \; y_{1}) = 1$, so the functional equation can be used to recursively compute our desired quantity $P_{n}(123, 132; \; s, t; \; 1 \text{ [n times]}; \; 1 \text{ [n times]}) = f_{n}(123, 132; \; s, t)$.

More generally, we can similarly extend $f_{n}(\tau; \; t)$ to $k$ different patterns $\tau_{1}, \tau_{2}, \ldots, \tau_{k}$ and the corresponding variables $t_{1}, t_{2}, \ldots, t_{k}$ as:
\begin{equation}
	f_{n}(\tau_{1}, \tau_{2}, \ldots, \tau_{k}; \; t_{1}, t_{2}, \ldots, t_{k}) := \mathop{\sum} \limits_{\pi \in \SG_{n}} {t_{1}^{N_{\tau_{1}}(\pi)} t_{2}^{N_{\tau_{2}}(\pi)} \ldots t_{k}^{N_{\tau_{k}}(\pi)}}.
\end{equation}

\noindent The generalized polynomials $P_{n}$ can be similarly defined and analogous functional equations can be derived.

For example, suppose that we want to consider all length three patterns simultaneously. We will consider the patterns in lexicographical order (i.e., $\tau_{1} = 123, \; \tau_{2} = 132, \; \ldots, \; \tau_{6} = 321$). Our $f_{n}$ polynomial now becomes:
\begin{equation}
	f_{n}(123, 132, \ldots, 321; \; t_{1}, t_{2}, \ldots, t_{6}) := \mathop{\sum} \limits_{\pi \in \SG_{n}} {t_{1}^{N_{123}(\pi)} t_{2}^{N_{132}(\pi)} \ldots t_{6}^{N_{321}(\pi)}}. \label{FS3}
\end{equation}

\noindent For notational convenience, the polynomial $f_{n}(123, 132, \ldots, 321; \; t_{1}, t_{2}, \ldots, t_{6})$ will be denoted by $f_{n}(\SG_{3}; \; t_{1}, \ldots, t_{6})$. In \cite{BN-GWILF2}, we discuss how to extend this to the generalized polynomial $P_{n}$ and derive analogous functional equations.

The previous polynomial could also be refined further to consider all length three patterns and the pattern $1234$ simultaneously. We will again consider the length three patterns in lexicographical order. Our $f_{n}$ polynomial now becomes:
\begin{equation}
	f_{n}(1234, \SG_{3}; \; s, t_{1}, t_{2}, \ldots, t_{6}) := \mathop{\sum} \limits_{\pi \in \SG_{n}} {s^{N_{1234}(\pi)} t_{1}^{N_{123}(\pi)} t_{2}^{N_{132}(\pi)} \ldots t_{6}^{N_{321}(\pi)}}. \label{FI4S3}
\end{equation}
Just like the previous case, this polynomial can be extended to the analogous generalized polynomial $P_{n}$ and similar functional equations can be derived.


\subsection{Adapting multi-pattern functional equations}

The previously described $f_{n}$ polynomials (and their corresponding generalized $P_{n}$ polynomials and functional equations) can be easily specialized to consider a variety of scenarios. This allows us to quickly extract functional equations (and fast enumeration algorithms) in a number of cases.

The polynomial $f_{n}(\SG_{3}; \; t_{1}, \ldots, t_{6})$ (in Eq.~\ref{FS3}) can be specialized to consider any subset of $\SG_{3}$ by setting some $t_{i}$ variables to $1$. For example, $f_{n}(\SG_{3}; \; t_{1}, t_{2}, 1, 1, 1, 1)$ would give us the polynomial tracking $123$ and $132$ simultaneously. Setting $t_{i} = 1$ for $3 \leq i \leq 6$ in the generalized polynomial $P_{n}$ and its functional equation would reproduce Theorem~\ref{THM123n132}. This approach actually allows us to quickly compute the bi-variate polynomial
\begin{equation}
	f_{n}(\sigma, \tau; \; s, t) = \mathop{\sum} \limits_{\pi \in \SG_{n}} {s^{N_{\sigma}(\pi)} t^{N_{\tau}(\pi)}}
\end{equation}

\noindent for any patterns $\sigma, \tau \in \SG_{3}$ (with $\sigma \neq \tau$).

The polynomial $f_{n}(\SG_{3}; \; t_{1}, \ldots, t_{6})$ can actually be specialized in other ways. Suppose that we wanted to compute the bi-variate polynomial
\begin{equation}
	\mathop{\sum} \limits_{\pi \in \SG_{n}(132)} {s^{N_{123}(\pi)} t^{N_{321}(\pi)}}.
\end{equation}

Observe that this is exactly $f_{n}(\SG_{3}; \; s, 0, 1, 1, 1, t)$. In other words, we may find the coefficient of $t_{2}^{0}$ in $f_{n}(\SG_{3}; \; t_{1}, \ldots, t_{6})$ and then set $t_{3} = t_{4} = t_{5} = 1$ and $t_{1} = s, t_{6} = t$. The same approach can be used to compute the polynomial
\begin{equation}
	\mathop{\sum} \limits_{\pi \in \SG_{n}(132)} {s^{N_{\sigma}(\pi)} t^{N_{\tau}(\pi)}}.
\end{equation}

\noindent for any patterns $\sigma, \tau \in \SG_{3} \backslash \{ 132 \}$ (with $\sigma \neq \tau$).

The analogous specialization can be done to quickly compute
\begin{equation}
	\mathop{\sum} \limits_{\pi \in \SG_{n}(123)} {s^{N_{\sigma}(\pi)} t^{N_{\tau}(\pi)}}.
\end{equation}

\noindent for any patterns $\sigma, \tau \in \SG_{3} \backslash \{ 123 \}$ (with $\sigma \neq \tau$). In general, for any $p \in \SG_{3}$, we can quickly compute 
\begin{equation}
	\mathop{\sum} \limits_{\pi \in \SG_{n}(p)} {s^{N_{\sigma}(\pi)} t^{N_{\tau}(\pi)}}.
\end{equation}

\noindent for any patterns $\sigma, \tau \in \SG_{3} \backslash \{ p \}$ (with $\sigma \neq \tau$).

We can also adapt the polynomial $f_{n}(1234, \SG_{3}; \; s, t_{1}, t_{2}, \ldots, t_{6})$ (from Eq.~\ref{FI4S3}) similarly. In particular, we can quickly compute the polynomial
\begin{equation}
	\mathop{\sum} \limits_{\pi \in \SG_{n}(1234)} {s^{N_{\sigma}(\pi)} t^{N_{\tau}(\pi)}}.
\end{equation}

\noindent for any patterns $\sigma, \tau \in \SG_{3}$ (with $\sigma \neq \tau$) by setting $s = 0$ (i.e.~extracting the coefficient of $s^{0}$) and setting the appropriate $t_{i}$'s to $1$ in $f_{n}(1234, \SG_{3}; \; s, t_{1}, t_{2}, \ldots, t_{6})$.

The previously discussed functional equation approaches have been implemented in the Maple packages {\tt PDSn}, {\tt PDAV132}, {\tt PDAV123}, and {\tt PDAV1234}.


\section{Computing moments for random permutations}\label{SECmoments}


\subsection{Moments for random permutations from $\SG_{n}$}\label{SECmomsSn}

The previously discussed functional equations approach allows us to compute both rigorous and empirical statistical properties on permutations.

For some fixed $n$ and fixed pattern $\sigma \in \SG_{k}$, suppose that a permutation $\pi \in \SG_{n}$ is chosen uniformly at random. Let the random variable $X_{\sigma}(\pi)$ be the number of occurrences of the pattern $\sigma$ in $\pi$. It is not hard to compute the expected value (i.e., the first moment of the random variable $X$): $\mathbb{E}[X] = {n \choose k}/k!$. More generally, it was shown in \cite{DZ-SMC1} that each of the higher moments of $X$ is a polynomial in $n$. In particular, the $r$-th moment about the mean of $X$, which is $\mathbb{E}[(X - \mathbb{E}[X])^{r}]$, is a polynomial of degree $\left\lfloor r(k - 1/2) \right\rfloor$ for $r \geq 2$.\footnote{This corrects a minor inaccuracy in \cite{DZ-SMC1}.}

For the patterns $\sigma$ that were discussed in the previous section, the functional equations approach allows us to quickly compute $f_{n}(\sigma; \; t)$ for any desired $n$. Observe that $f_{n}(\sigma; \; t)/n!$ gives us the polynomial where the coefficient of $t^{i}$ is the probability that a randomly chosen $\pi \in \SG_{n}$ will have exactly $i$ copies of $\sigma$. The important point is that we can (rigorously) find a closed-form expression (in $n$) for the higher order moments of $X$ by computing sufficiently many terms to fit the polynomial.

For example, it was shown in \cite{DZ-SMC1} that the exact expression for the second moment (about the mean) of the random variable $X_{123}$ (over $\SG_{n}$) is:
\begin{equation}
	\frac{n (n-1) (n-2) (39 n^{2} + 102 n - 157)}{21600}
\end{equation}

\noindent and that the third moment (about the mean) of the random variable $X_{123}$ (over $\SG_{n}$) is:
\begin{equation}
	\frac{n (n-1) (n-2) (1437 n^{4} + 5592 n^{3} - 11277 n^{2} - 33990 n + 34082)}{6350400}
\end{equation}

Similarly, the exact expression for the second moment (about the mean) of the random variable $X_{132}$ (over $\SG_{n}$) is:
\begin{equation}
	\frac{n (n-1) (n-2) (21 n^{2} + 78 n + 77)}{21600}
\end{equation}

\noindent and that the third moment (about the mean) of the random variable $X_{132}$ (over $\SG_{n}$) is:
\begin{equation}
	\frac{n (n-1) (n-2) (129 n^{4} + 3705 n^{3} + 5355 n^{2} + 8655 n + 11356)}{12700800}
\end{equation}

We may also consider mixed moments for two patterns $\sigma$ and $\tau$. Suppose that a permutation $\pi$ is chosen uniformly at random from $\SG_{n}$, and again let the random variable $X_{\sigma}(\pi)$ be the number of occurrences of pattern $\sigma$ in $\pi$ (and equivalently for $X_{\tau}(\pi)$). It was also shown in \cite{DZ-SMC1} that the mixed moments of the random variables $X_{\sigma}$ and $X_{\tau}$ (about their respective means) are also polynomials in $n$. This allows us to rigorously find closed-form expressions (in $n$) for the higher order mixed moments by computing enough terms to find the polynomial.

For example, the covariance of the two random variables $X_{123}$ and $X_{132}$ is:
\begin{equation}
	\frac{n (n-1) (n-2) (18 n^{2} - 51 n - 109)}{21600}
\end{equation}

\noindent while the covariance of the two random variables $X_{123}$ and $X_{312}$ is:
\begin{equation}
	- \frac{n (n-1) ( n-2) (39 n^{2} - 48 n - 7)}{43200}
\end{equation}

\noindent and the covariance of the two random variables $X_{123}$ and $X_{321}$ is:
\begin{equation}
	- \frac{n (n-1) (n-2) (9 n^{2} + 12 n - 92)}{5400}
\end{equation}

\noindent Similar results for other random variables can be derived using the Maple packages available on the authors' website.


\subsection{Moments for random permutations from $\SG_{n}(\tau)$}

There has been a flurry of recent activity studying occurrences of patterns in the set of permutations avoiding specific patterns. Many of the recent articles focus on counting the total number of occurrences of a pattern in $\SG_{n}(132)$ or in $\SG_{n}(123)$. Some examples (as previously mentioned) include \cite{Bona2, Bona4, Homberger, Rudolph}. It is important to note that finding the total number of occurrences of pattern $\sigma$ in the set $\SG_{n}(\tau)$ is equivalent to picking a permutation uniformly at random from $\SG_{n}(\tau)$ and finding the expected value $\mathbb{E}[X_{\sigma}]$ (assuming that the enumeration of $\SG_{n}(\tau)$ is known).

In the previous section, we were able to rigorously derive closed-form
expressions for moments of the random variable $X_{\sigma}(\pi)$ when the
permutation $\pi$ was randomly chosen from $\SG_{n}$. While we currently
cannot derive similar rigorous results for random permutations from
$\SG_{n}(\tau)$, we can still compute numerical moments for a variety of
cases. Interestingly, a number of such random variables appear to \emph{not}
be asymptotically normal (as opposed to when $\pi \in \SG_{n}$, where
Mikl\'os B\'ona showed that such random variables are asymptotically normal
\cite{Bona}, see also Section \ref{SECasymom}).

\FloatBarrier

\subsubsection{Permutations from $\SG_{132}$}

Suppose a permutation is chosen uniformly at random from $\SG_{n}(132)$. Using the Maple packages that accompany this article, we can compute many empirical moments. The expected values of the random variables $X_{123}$, $X_{312}$, and $X_{321}$ for $1 \leq n \leq 10$ can be found in Table~\ref{tab:t132mom1}.

\begin{table}[!h] 
	\centering
	\begin{tabular}{|c|r|r|r|r|r|r|r|r|r|r|}
		\hline
		Pattern & \cheader{$n=1$} & \cheader{$n=2$} & \cheader{$n=3$} & \cheader{$n=4$} & \cheader{$n=5$} & \cheader{$n=6$} & \cheader{$n=7$} & \cheader{$n=8$} & \cheader{$n=9$} & \cheader{$n=10$}\\
		\hline
		$123$ & $0$ & $0$ & $0.200$ & $0.714$ & $1.619$ & $2.970$ & $4.809$ & $7.171$ & $10.083$ & $13.570$\\
		\hline
		$312$ & $0$ & $0$ & $0.200$ & $0.786$ & $1.929$ & $3.790$ & $6.513$ & $10.244$ & $15.115$ & $21.253$\\
		\hline
		$321$ & $0$ & $0$ & $0.200$ & $0.929$ & $2.595$ & $5.667$ & $10.653$ & $18.097$ & $28.572$ & $42.672$\\
		\hline
	\end{tabular}
	\caption{Expected values (first moments) of $X_{123}(\pi)$, $X_{312}(\pi)$, and $X_{321}(\pi)$, where $\pi$ is chosen uniformly at random from $\SG_{n}(132)$.}
	\label{tab:t132mom1}
\end{table}

\medskip

The second moments (about the mean) of the random variables $X_{123}$, $X_{312}$, and $X_{321}$ for $1 \leq n \leq 10$ can be found in Table~\ref{tab:t132mom2}.

\begin{table}[!h] 
	\centering
	\begin{tabular}{|c|r|r|r|r|r|r|r|r|r|r|}
		\hline
		Pattern & \cheader{$n=1$} & \cheader{$n=2$} & \cheader{$n=3$} & \cheader{$n=4$} & \cheader{$n=5$} & \cheader{$n=6$} & \cheader{$n=7$} & \cheader{$n=8$} & \cheader{$n=9$} & \cheader{$n=10$}\\
		\hline
		$123$ & $0$ & $0$ & $0.160$ & $1.204$ & $4.617$ & $12.757$ & $28.933$ & $57.463$ & $103.720$ & $174.140$\\
		\hline
		$312$ & $0$ & $0$ & $0.160$ & $1.026$ & $3.733$ & $10.213$ & $23.392$ & $47.403$ & $87.787$ & $151.710$\\
		\hline
		$321$ & $0$ & $0$ & $0.160$ & $1.352$ & $6.003$ & $19.101$ & $49.313$ & $110.180$ & $221.360$ & $409.960$\\
		\hline
	\end{tabular}
	\caption{Second moments (about the mean) of $X_{123}(\pi)$, $X_{312}(\pi)$, and $X_{321}(\pi)$, where $\pi$ is chosen uniformly at random from $\SG_{n}(132)$.}
	\label{tab:t132mom2}
\end{table}

\medskip

Data for the higher moments can be found on the authors websites. For example, the $r$-th standardized moments for $X_{312}$ when $3 \leq r \leq 6$ and $15 \leq n \leq 20$ can be found in Table~\ref{tab:t132nonnorm}.

\begin{table}[!h] 
	\centering
	\begin{tabular}{|c|r|r|r|r|r|r|r|r|r|r|}
		\hline
		$r$-th moment & \cheader{$n=15$} & \cheader{$n=16$} & \cheader{$n=17$} & \cheader{$n=18$} & \cheader{$n=19$} & \cheader{$n=20$}\\
		\hline
		$r=3$ & $0.41867$ & $0.42461$ & $0.43073$ & $0.43690$ & $0.44303$ & $0.44906$\\
		\hline
		$r=4$ & $2.92652$ & $2.95682$ & $2.98412$ & $3.00889$ & $3.03152$ & $3.05231$\\
		\hline
		$r=5$ & $3.59958$ & $3.69377$ & $3.78619$ & $3.87633$ & $3.96389$ & $4.04860$\\
		\hline
		$r=6$ & $14.79293$ & $15.24562$ & $15.66679$ & $16.06007$ & $16.42853$ & $16.77483$\\
		\hline
	\end{tabular}
	\caption{$r$-th standardized moments for $X_{312}(\pi)$ for $3 \leq r \leq 6$, where $\pi$ is chosen uniformly at random from $\SG_{n}(132)$.}
	\label{tab:t132nonnorm}
\end{table}

\medskip

It is interesting to note that the random variable $X_{312}$ does not appear to be asymptotically normal since the $3$-rd and $5$-th standard moments appear to be increasing (as opposed to going to $0$ as a normal distribution would) and the $6$-th moment appears to be larger than $15$ (the value for a normal distribution).

This approach can also be used to consider the mixed $(i,j)$ moments. For example, the mixed $(i,j)$ moments of the random variables $X_{123}$ and $X_{321}$ for $3 \leq n \leq 10$ can be found in Table~\ref{tab:t132mixmom}.

\begin{table}[!h] 
	\centering
	\begin{tabular}{|c|r|r|r|r|r|r|r|r|r|r|}
		\hline
		$(i,j)$ & \cheader{$n=3$} & \cheader{$n=4$} & \cheader{$n=5$} & \cheader{$n=6$} & \cheader{$n=7$} & \cheader{$n=8$} & \cheader{$n=9$} & \cheader{$n=10$}\\
		\hline
		$(1,1)$ & $-0.040$ & $-0.663$ & $-3.392$ & $-11.162$ & $-28.714$ & $-62.970$ & $-123.370$ & $-222.180$\\
		\hline
		$(1,2)$ & $-0.024$ & $-0.350$ & $-1.445$ & $-0.404$ & $21.587$ & $127.800$ & $478.610$ & $1417.300$\\
		\hline
		$(2,1)$ & $-0.024$ & $-0.644$ & $-6.657$ & $-38.272$ & $-154.230$ & $-491.000$ & $-1322.000$ & $-3140.400$\\
		\hline
		$(2,2)$ & $0.011$ & $1.288$ & $33.666$ & $382.200$ & $2650.400$ & $13264.000$ & $52628.000$ & $175500.000$\\
		\hline
	\end{tabular}
	\caption{Mixed $(i,j)$ moments of $X_{123}(\pi)$ and $X_{321}(\pi)$, where $\pi$ is chosen uniformly at random from $\SG_{n}(132)$.}
	\label{tab:t132mixmom}
\end{table}

\medskip

Analogous data and outputs can be found on the authors websites.

\FloatBarrier

\subsubsection{Permutations from $\SG_{123}$}

Suppose a permutation is chosen uniformly at random from $\SG_{n}(123)$. Using the Maple packages that accompany this article, we can compute many empirical moments. The expected values of the random variables $X_{132}$, $X_{312}$, and $X_{321}$ for $1 \leq n \leq 10$ can be found in Table~\ref{tab:t123mom1}.

\begin{table}[!h] 
	\centering
	\begin{tabular}{|c|r|r|r|r|r|r|r|r|r|r|}
		\hline
		Pattern & \cheader{$n=1$} & \cheader{$n=2$} & \cheader{$n=3$} & \cheader{$n=4$} & \cheader{$n=5$} & \cheader{$n=6$} & \cheader{$n=7$} & \cheader{$n=8$} & \cheader{$n=9$} & \cheader{$n=10$}\\
		\hline
		$132$ & $0$ & $0$ & $0.200$ & $0.643$ & $1.357$ & $2.364$ & $3.678$ & $5.314$ & $7.281$ & $9.589$\\
		\hline
		$312$ & $0$ & $0$ & $0.200$ & $0.786$ & $1.929$ & $3.788$ & $6.513$ & $10.244$ & $15.115$ & $21.253$\\
		\hline
		$321$ & $0$ & $0$ & $0.200$ & $1.143$ & $3.429$ & $7.697$ & $14.618$ & $24.884$ & $39.208$ & $58.317$\\
		\hline
	\end{tabular}
	\caption{Expected values (first moments) of $X_{132}(\pi)$, $X_{312}(\pi)$, and $X_{321}(\pi)$, where $\pi$ is chosen uniformly at random from $\SG_{n}(123)$.}
	\label{tab:t123mom1}
\end{table}

\medskip

The second moments (about the mean) of the random variables $X_{132}$, $X_{312}$, and $X_{321}$ for $1 \leq n \leq 10$ can be found in Table~\ref{tab:t123mom2}.

\begin{table}[!h] 
	\centering
	\begin{tabular}{|c|r|r|r|r|r|r|r|r|r|r|}
		\hline
		Pattern & \cheader{$n=1$} & \cheader{$n=2$} & \cheader{$n=3$} & \cheader{$n=4$} & \cheader{$n=5$} & \cheader{$n=6$} & \cheader{$n=7$} & \cheader{$n=8$} & \cheader{$n=9$} & \cheader{$n=10$}\\
		\hline
		$132$ & $0$ & $0$ & $0.160$ & $0.801$ & $2.468$ & $5.959$ & $12.344$ & $22.978$ & $39.506$ & $63.877$\\
		\hline
		$312$ & $0$ & $0$ & $0.160$ & $0.740$ & $2.114$ & $4.804$ & $9.532$ & $17.303$ & $29.501$ & $48.000$\\
		\hline
		$321$ & $0$ & $0$ & $0.160$ & $1.122$ & $4.293$ & $12.423$ & $30.287$ & $65.419$ & $128.910$ & $236.250$\\
		\hline
	\end{tabular}
	\caption{Second moments (about the mean) of $X_{132}(\pi)$, $X_{312}(\pi)$, and $X_{321}(\pi)$, where $\pi$ is chosen uniformly at random from $\SG_{n}(123)$.}
	\label{tab:t123mom2}
\end{table}

\medskip

Data for the higher moments can be found on the authors websites. For example, the $r$-th standardized moments for $X_{132}$ when $3 \leq r \leq 6$ and $15 \leq n \leq 20$ can be found in Table~\ref{tab:t123nonnorm}.

\begin{table}[!h] 
	\centering
	\begin{tabular}{|c|r|r|r|r|r|r|r|r|r|r|}
		\hline
		$r$-th moment & \cheader{$n=15$} & \cheader{$n=16$} & \cheader{$n=17$} & \cheader{$n=18$} & \cheader{$n=19$} & \cheader{$n=20$}\\
		\hline
		$r=3$ & $1.53492$ & $1.54020$ & $1.54458$ & $1.54823$ & $1.55129$ & $1.55385$\\
		\hline
		$r=4$ & $6.28717$ & $6.33967$ & $6.38469$ & $6.42356$ & $6.45735$ & $6.48687$\\
		\hline
		$r=5$ & $23.59568$ & $23.99423$ & $24.34048$ & $24.64315$ & $24.90923$ & $25.14433$\\
		\hline
		$r=6$ & $108.90240$ & $111.90699$ & $114.55548$ & $116.90184$ & $118.99022$ & $120.85698$\\
		\hline
	\end{tabular}
	\caption{$r$-th standardized moments for $X_{132}(\pi)$ for $3 \leq r \leq 6$, where $\pi$ is chosen uniformly at random from $\SG_{n}(123)$.}
	\label{tab:t123nonnorm}
\end{table}

\medskip

It is interesting to note that the random variable $X_{132}$ does not appear to be asymptotically normal since the $3$-rd and $5$-th standard moments appear to be increasing (as opposed to going to $0$ as a normal distribution would), the $4$-th moment appears to be larger than $3$ (the value for a normal distribution), and the $6$-th moment appears to be substantially larger than $15$ (the value for a normal distribution).

This approach can also be used to consider the mixed $(i,j)$ moments. For example, the mixed $(i,j)$ moments of the random variables $X_{132}$ and $X_{312}$ for $3 \leq n \leq 10$ can be found in Table~\ref{tab:t123mixmom}.

\begin{table}[!h] 
	\centering
	\begin{tabular}{|c|r|r|r|r|r|r|r|r|r|r|}
		\hline
		$(i,j)$ & \cheader{$n=3$} & \cheader{$n=4$} & \cheader{$n=5$} & \cheader{$n=6$} & \cheader{$n=7$} & \cheader{$n=8$} & \cheader{$n=9$} & \cheader{$n=10$}\\
		\hline
		$(1,1)$ & $-0.040$ & $-0.219$ & $-0.641$ & $-1.362$ & $-2.332$ & $-3.326$ & $-3.890$ & $-3.269$\\
		\hline
		$(1,2)$ & $-0.024$ & $-0.099$ & $-0.039$ & $0.841$ & $3.917$ & $11.254$ & $25.372$ & $48.890$\\
		\hline
		$(2,1)$ & $-0.024$ & $-0.386$ & $-2.261$ & $-8.566$ & $-24.874$ & $-60.099$ & $-126.620$ & $-239.570$\\
		\hline
		$(2,2)$ & $0.011$ & $0.551$ & $6.309$ & $39.592$ & $172.880$ & $592.420$ & $1709.800$ & $4350.100$\\
		\hline
	\end{tabular}
	\caption{Mixed $(i,j)$ moments of $X_{132}(\pi)$ and $X_{312}(\pi)$, where $\pi$ is chosen uniformly at random from $\SG_{n}(123)$.}
	\label{tab:t123mixmom}
\end{table}

\medskip

Analogous data and outputs can be found on the authors websites.

\FloatBarrier
\subsubsection{Permutations from $\SG_{1234}$}

Suppose a permutation is chosen uniformly at random from $\SG_{n}(1234)$. Using the Maple packages that accompany this article, we can compute many empirical moments. The expected values of the random variables $X_{123}$, $X_{132}$, $X_{312}$, and $X_{321}$ for $1 \leq n \leq 10$ can be found in Table~\ref{tab:t1234mom1}.

\begin{table}[!h] 
	\centering
	\begin{tabular}{|c|r|r|r|r|r|r|r|r|r|r|}
		\hline
		Pattern & \cheader{$n=1$} & \cheader{$n=2$} & \cheader{$n=3$} & \cheader{$n=4$} & \cheader{$n=5$} & \cheader{$n=6$} & \cheader{$n=7$} & \cheader{$n=8$} & \cheader{$n=9$} & \cheader{$n=10$}\\
		\hline
		$123$ & $0$ & $0$ & $0.167$ & $0.522$ & $1.049$ & $1.739$ & $2.592$ & $3.611$ & $4.796$ & $6.153$\\
		\hline
		$132$ & $0$ & $0$ & $0.167$ & $0.696$ & $1.709$ & $3.279$ & $5.457$ & $8.283$ & $11.789$ & $16.004$\\
		\hline
		$312$ & $0$ & $0$ & $0.167$ & $0.696$ & $1.796$ & $3.684$ & $6.575$ & $10.679$ & $16.202$ & $23.341$\\
		\hline
		$321$ & $0$ & $0$ & $0.167$ & $0.696$ & $1.942$ & $4.335$ & $8.344$ & $14.466$ & $23.223$ & $35.158$\\
		\hline
	\end{tabular}
	\caption{Expected values (first moments) of $X_{123}(\pi)$, $X_{132}(\pi)$, $X_{312}(\pi)$, and $X_{321}(\pi)$, where $\pi$ is chosen uniformly at random from $\SG_{n}(1234)$.}
	\label{tab:t1234mom1}
\end{table}

\medskip

The second moments (about the mean) of the random variables $X_{123}$, $X_{312}$, and $X_{321}$ for $1 \leq n \leq 10$ can be found in Table~\ref{tab:t1234mom2}.

\begin{table}[!h] 
	\centering
	\begin{tabular}{|c|r|r|r|r|r|r|r|r|r|r|}
		\hline
		Pattern & \cheader{$n=1$} & \cheader{$n=2$} & \cheader{$n=3$} & \cheader{$n=4$} & \cheader{$n=5$} & \cheader{$n=6$} & \cheader{$n=7$} & \cheader{$n=8$} & \cheader{$n=9$} & \cheader{$n=10$}\\
		\hline
		$123$ & $0$ & $0$ & $0.139$ & $0.510$ & $1.172$ & $2.236$ & $3.863$ & $6.257$ & $9.654$ & $14.324$\\
		\hline
		$132$ & $0$ & $0$ & $0.139$ & $0.820$ & $2.828$ & $7.332$ & $15.959$ & $30.863$ & $54.767$ & $91.002$\\
		\hline
		$312$ & $0$ & $0$ & $0.139$ & $0.820$ & $2.667$ & $6.524$ & $13.484$ & $24.911$ & $42.468$ & $68.157$\\
		\hline
		$321$ & $0$ & $0$ & $0.139$ & $0.994$ & $3.764$ & $10.566$ & $24.936$ & $52.338$ & $100.740$ & $181.280$\\
		\hline
	\end{tabular}
	\caption{Second moments (about the mean) of $X_{123}(\pi)$, $X_{132}(\pi)$, $X_{312}(\pi)$, and $X_{321}(\pi)$, where $\pi$ is chosen uniformly at random from $\SG_{n}(1234)$.}
	\label{tab:t1234mom2}
\end{table}

\medskip

Data for the higher moments can be found on the authors websites. For example, the $r$-th standardized moments for $X_{123}$ when $3 \leq r \leq 6$ and $13 \leq n \leq 18$ can be found in Table~\ref{tab:t1234nonnorm}.

\begin{table}[!h] 
	\centering
	\begin{tabular}{|c|r|r|r|r|r|r|r|r|r|r|}
		\hline
		$r$-th moment & \cheader{$n=13$} & \cheader{$n=14$} & \cheader{$n=15$} & \cheader{$n=16$} & \cheader{$n=17$} & \cheader{$n=18$}\\
		\hline
		$r=3$ & $1.14140$ & $1.16076$ & $1.17518$ & $1.18585$ & $1.19365$ & $1.19926$\\
		\hline
		$r=4$ & $5.14732$ & $5.21356$ & $5.26297$ & $5.29971$ & $5.32683$ & $5.34656$\\
		\hline
		$r=5$ & $16.61123$ & $17.07925$ & $17.43934$ & $17.71522$ & $17.92523$ & $18.08348$\\
		\hline
		$r=6$ & $74.59126$ & $77.40043$ & $79.60569$ & $81.33022$ & $82.67201$ & $83.70841$\\
		\hline
	\end{tabular}
	\caption{$r$-th standardized moments for $X_{123}(\pi)$ for $3 \leq r \leq 6$, where $\pi$ is chosen uniformly at random from $\SG_{n}(1234)$.}
	\label{tab:t1234nonnorm}
\end{table}

\medskip

It is interesting to note that the random variable $X_{123}$ does not appear to be asymptotically normal since the $3$-rd and $5$-th standard moments appear to be increasing (as opposed to going to $0$ as a normal distribution would), the $4$-th moment appears to be larger than $3$ (the value for a normal distribution), and the $6$-th moment appears to be substantially larger than $15$ (the value for a normal distribution).

This approach can also be used to consider the mixed $(i,j)$ moments. For example, the mixed $(i,j)$ moments of the random variables $X_{123}$ and $X_{321}$ for $3 \leq n \leq 10$ can be found in Table~\ref{tab:t1234mixmom}.

\begin{table}[!h] 
	\centering
	\begin{tabular}{|c|r|r|r|r|r|r|r|r|r|r|}
		\hline
		$(i,j)$ & \cheader{$n=3$} & \cheader{$n=4$} & \cheader{$n=5$} & \cheader{$n=6$} & \cheader{$n=7$} & \cheader{$n=8$} & \cheader{$n=9$} & \cheader{$n=10$}\\
		\hline
		$(1,1)$ & $-0.028$ & $-0.363$ & $-1.298$ & $-3.258$ & $-6.892$ & $-13.121$ & $-23.171$ & $-38.611$\\
		\hline
		$(1,2)$ & $-0.019$ & $-0.266$ & $-1.674$ & $-5.958$ & $-15.301$ & $-31.716$ & $-55.546$ & $-82.648$\\
		\hline
		$(2,1)$ & $-0.019$ & $-0.166$ & $-0.505$ & $-1.531$ & $-4.798$ & $-13.664$ & $-34.352$ & $-77.387$\\
		\hline
		$(2,2)$ & $0.007$ & $0.386$ & $4.969$ & $33.937$ & $159.600$ & $593.990$ & $1880.700$ & $5274.100$\\
		\hline
	\end{tabular}
	\caption{Mixed $(i,j)$ moments of $X_{123}(\pi)$ and $X_{321}(\pi)$, where $\pi$ is chosen uniformly at random from $\SG_{n}(1234)$.}
	\label{tab:t1234mixmom}
\end{table}

\medskip

Analogous data and outputs can be found on the authors websites.

\FloatBarrier
\section{Joint asymptotic normality of multiple patterns}\label{SECasymom}

In this section we let $\pi$ be a permutation
chosen uniformly at random from $\SG_n$
(without any condition) and we study the joint distribution of the random
variables $\xn\gs:=X_\gs(\pi)$, the number of copies of $\gs$ in $\pi$, for
different patterns $\gs\in\SG_* :=\bigcup_{k=1}^\infty \SG_k$ .
We consider asymptotics as $n\to\infty$ for (one or several) fixed $\gs$.

Each $\xn\gs$ has an asymptotic normal distribution, as was shown by Bona
\cite{Bona} (see also \cite{Bona3}).
We give another (perhaps simpler) proof of this; moreover, we extend the
result to joint asymptotic normality for several patterns $\gs$.

The asymptotic variances and covariances depend on the patterns in a 
slightly complicated way, so we begin with some  definitions.
For $k\ge 1$ and $1\le i \le k$,
define
\begin{equation}\label{g}
  g_{k,i}(x) := \binom{k-1}{i-1} x^{i-1} (1-x)^{k-i}.
\end{equation}

For a permutation $\gs\in S_k$, define
\begin{equation}\label{G}
  G_\gs(x,y) := \frac{1}{(k-1)!^2} 
\left(\sum_{i=1}^k g_{k,i}(x) g_{k,\gs(i)}(y) -\frac 1k \right).
\end{equation}

Let $Z_\gs$, $\gs\in\SG_*$, be jointly normal random variables with
$\E Z_\gs=0$ and (co)variances
\begin{equation}\label{gS}
  \Cov(Z_\gs,Z_\gt) = 
\gS_{\gs,\gt}:=
\innprod{G_\gs,G_\gt}_{L^2(\oi^2)}
:=\intoi\intoi G_\gs(x,y) G_\gt(x,y)\dd x\dd y.
\end{equation}
(Such normal random variables exist since the matrix
$(\gS_{\gs,\gt})_{\gs,\gt}$ is non-negative definite. As is well
known, the joint distribution is uniquely defined by the means and covariances.)

We denote the length of a permutation $\gs$ by $|\gs|$, and let $\dto$ denote
convergence in distribution of random variables.

\begin{theorem}\label{T1}

For every pattern $\gs\in\SG_*$, as $n\to\infty$,
\begin{equation}\label{t1a}
	\frac{\xn\gs-\E\xn\gs}{n^{|\gs|-1/2} }
	=
	\frac{\xn\gs-\frac1{|\gs|!}\binom{n}{|\gs|}}{n^{|\gs|-1/2} }
	\dto Z_\gs.
\end{equation}

Moreover, this holds jointly for any finite family of patterns $\gs$.
Furthermore, all (joint) moments converge; 
in particular, for any permutations $\gs,\gt$
\begin{equation}\label{t1b}
 \frac{ \Cov(\xn\gs,\xn\gt)}{n^{|\gs|+|\gt|-1}} \to \gS_{\gs,\gt}.
\end{equation}

\end{theorem}

Before giving the proof we give some comments.
First, as noted above, 
if $\gs$ has length $|\gs|=k$,
\begin{equation}
  \E \xn\gs = \binom nk \frac1{k!}
	\sim \frac1{k!^2} n^k,
	\qquad
	\text{as $n\to\infty$}.
\end{equation}

The asymptotic covariances $\gS_{\gs,\gt}$ can be computed explicitly.
 By a beta integral,
\begin{equation}\label{intg}
  \int_0^1 g_{k,i}(x)\, dx
	=
	\binom{k-1}{i-1} \frac{\Gamma(i)\Gamma(k-i+1)}{\Gamma(k+1)}
	=\frac 1k,
\end{equation}
and similarly, for any $k,\ell\ge1$ and $1\le i\le k$, $1\le j\le\ell$,
\begin{equation} \label{intgg}
  \begin{split}
  \int_0^1 g_{k,i}(x) g_{\ell,j}(x)\, dx
	&=
	\binom{k-1}{i-1} \binom{\ell-1}{j-1} 
	\frac{\Gamma(i+j-1)\Gamma(k+\ell-i-j+1)}{\Gamma(k+\ell)}
	\\&
	= \frac{(k-1)!\, (\ell-1)!}{(k+\ell-1)!}
	\binom{i+j-2}{i-1} \binom{k+\ell-i-j}{k-i} .
  \end{split}
\end{equation}

It follows from \eqref{intg} that, if $|\gs|=k$,
\begin{equation}
  \begin{split}
	\intoi\intoi
	\sum_{i=1}^k g_{k,i}(x) g_{k,\gs(i)}(y)
	\dd x\dd y
	= \frac{k}{k^2}=\frac1{k}
  \end{split}
\end{equation}
which implies, using \eqref{intgg} twice, if further $|\gt|=\ell$,
\begin{multline*}
	\intoi\intoi
	\left(\sum_{i=1}^k g_{k,i}(x) g_{k,\gs(i)}(y) -\frac 1k \right)
	\left(\sum_{j=1}^\ell g_{\ell,j}(x) g_{\ell,\gt(j)}(y) -\frac 1\ell \right)
	\dd x\dd y
	\\
	=
	\intoi\intoi
	\sum_{i=1}^k g_{k,i}(x) g_{k,\gs(i)}(y) 
	\sum_{j=1}^\ell g_{\ell,j}(x) g_{\ell,\gt(j)}(y) 
	\dd x\dd y -\frac{1}{k\ell}  
	\\
	=
	\sum_{i=1}^k \sum_{j=1}^\ell 
	\intoi
	g_{k,i}(x) 
	g_{\ell,j}(x) 
	\dd x
	\intoi
	g_{k,\gs(i)}(y) 
	g_{\ell,\gt(j)}(y) 
	\dd y -\frac{1}{k\ell}  
	\\
	=
	\sum_{i=1}^k \sum_{j=1}^\ell 
	\frac{(k-1)!^2\, (\ell-1)!^2}{(k+\ell-1)!^2}
	\binom{i+j-2}{i-1} \binom{k+\ell-i-j}{k-i} 
	\binom{\gs(i)+\gt(j)-2}{\gs(i)-1} \binom{k+\ell-\gs(i)-\gt(j)}{k-\gs(i)} 
	\\
	-\frac{1}{k\ell}  .
\end{multline*}

Consequently, by \eqref{gS} and \eqref{G},
if $|\gs|=k$ and $|\gt|=\ell$, then
{\multlinegap=0pt

\begin{multline}\label{gS2}
	\gS_{\gs,\gt}
	=
	\frac{1}{(k+\ell-1)!^2}
	\sum_{i=1}^k\sum_{j=1}^\ell
	\binom{i+j-2}{i-1} \binom{k+\ell-i-j}{k-i} 
	\binom{\gs(i)+\gt(j)-2}{\gs(i)-1} \binom{k+\ell-\gs(i)-\gt(j)}{k-\gs(i)} 
	\\
	-\frac1{(k-1)!\,k!\,(\ell-1)!\,\ell!}.
\end{multline}}%

\begin{proof}[Proof of Theorem \ref{T1}]

Let $U_1,\dots,U_n$ be independent and identically distributed (i.i.d.)\
random variables with a uniform distribution on $\oi$.
It is a standard trick  that (by symmetry) the reduction
$\red(U_1,\dots,U_n)$ is a 
uniformly random permutation in $\SG_n$ (note that $U_1,\dots,U_n$ almost
surely are distinct), so we can take this as our random $\pi$ and obtain the
representation, with $k=|\gs|$,
\begin{equation}
  \label{rep1}
	\xn\gs = X_\gs(\pi)=\sum_{i_1<\dots<i_k}\ett{\red(U_{i_1},\dots,U_{i_k})=\gs}.
\end{equation}

This is an example of an asymmetric $U$-statistic, and 
(a rather simple instance of)
the general theory in
\cite[Section 11.2]{SJIII} 
can be used to show the theorem.
However, the details are a bit technical, in particular  to calculate the
asymptotic covariances, so we will instead use another, more symmetric
representation. (See \cite[Remark 11.21]{SJIII}.)

Let $V_1,\dots,V_n$ be another sequence of i.i.d.\ random variables,
uniformly distributed on \oi{} and independent of $U_1,\dots,U_n$. 
Let $\pi'$ be the permutation that sorts these
numbers such that $V_{\pi'(1)}<\dots<V_{\pi'(n)}$
and let $\pi$ be the reduction of $U_{\pi'(1)},\dots,U_{\pi'(n)}$.
Then $\pi$ is still uniformly random, and it is easy to see that
\begin{equation}
  \label{rep2}
	\begin{split}
	\xn\gs = X_\gs(\pi)
	&:=
	\sum_{i_1<\dots<i_k}
	\ett{\red(U_{\pi'(i_1)},\dots,U_{\pi'(i_k)})=\gs}
	\\&\phantom:
	=
	\sumx_{j_1,\dots,j_k}
	\ett{\red(U_{j_1},\dots,U_{j_k})=\gs}\cdot\ett{V_{j_1}<\dots<V_{j_k}},
	\end{split}
\end{equation}
where $\sumx$ denotes summation over all distinct indices
 $j_1,\dots,j_k$.
This representation, while in some ways more complicated that \eqref{rep1},
has the great advantage that we sum over all ordered $n$-tuples of
distinct indices;
this is thus an example of a $U$-statistic, and we can apply the
basic central limit theorem by Hoeffding \cite[Theorem 7.1]{Hoeffding},
see also \cite{Rubin-Vitale} and \cite[Section 11.1]{SJIII}.
In order to compute the (co)variances, we follow the path of
Hoeffding's proof.

The main idea of Hoeffding's proof of his
central limit theorem is to use a projection.
In our case we let $W_j:=(U_j,V_j)\in\oi^2$ and write \eqref{rep2} as
\begin{equation}\label{rep3}
  \xn\gs = \sumx_{j_1,\dots,j_k} f_\gs(W_{j_1},\dots,W_{j_k}),
\end{equation}
for a certain (indicator) function $f_\gs$.
We then take the conditional expectation of $f_\gs(W_1,\dots,W_k)$ given one
of the variables $W_i$:
\begin{equation}\label{fgs}
  f_{\gs;i}(x,y) := \E \bigpar{f_\gs(W_1,\dots,W_k)\mid W_i=(x,y)};
\end{equation}
we also take the expectation
\begin{equation}
	\mu := \E f_\gs(W_1,\dots,W_k) = \E f_{\gs;i}(W_i).  
\end{equation}

Hoeffding then shows that if we replace $f_\gs$ by 
$f'_\gs(W_1,\dots,W_k):=\mu+\sum_{i=1}^k (f_{\gs,i}(W_i)-\mu)$, 
then the resulting error for the sum in \eqref{rep3} will have variance 
$O(n^{2k-2})$, which is negligible with the normalization used in 
\refT{T1}.
Thus we can approximate
$\xn\gs-\E\xn\gs$ by
\begin{equation}
  \sumx_{j_1,\dots,j_n} \sum_{i=1}^k \bigpar{f_{\gs;i}(W_{j_i})-\mu}
	= \sum_{i=1}^k  (n-1)\fall{k-1}\sum_{j=1}^n\bigpar{f_{\gs;i}(W_{j})-\mu}
	=(n-1)\fall{k-1}\sum_{j=1}^n F_\gs(W_j),
\end{equation}
where $(n-1)\fall{k-1}=(n-1)\dotsm(n-k+1)$ and
\begin{equation}\label{F}
	F_\gs(x,y):=\sum_{i=1}^k \bigpar{f_{\gs;i}(x,y)-\mu}.
\end{equation}

The asymptotic normality of $\xn\gs$ now
follows by the standard central limit theorem
for the i.i.d.\ random variables $F_\gs(W_j)$,
which yields $(\xn\gs-\E\xn\gs)/n^{k-1/2}\dto N\bigpar{0,\gS_{\gs,\gs}}$ 
where 
\begin{equation}
  \label{gSF}
	\gS_{\gs,\gs}:=\E \bigpar{F_{\gs}(W_1)^2} 
	= \intoi\intoi F_{\gs}(x,y)^2\dd x\dd y.
\end{equation}

Joint normality for several patterns $\gs$ (possibly of different lengths)
follows in the same way, with the asymptotic covariances
\begin{equation}
	\label{gSF2}
	\gS_{\gs,\gt}:=\E \bigpar{F_{\gs}(W_1)F_\gt(W_1)} 
	= \intoi\intoi F_{\gs}(x,y)F_\gt(x,y)\dd x\dd y.
\end{equation}

It remains to compute the functions $F_\gs$ defined in \eqref{F} 
In order to do this, we see that from \eqref{fgs} and the definition of
$f_\gs$ as an indicator function, cf.\ \eqref{rep2}--\eqref{rep3},
\begin{equation}\label{sw}
  f_{\gs;i}(x,y) =
	\PP\bigpar{\red(U_{1},\dots,U_{k})=\gs\mid U_i=x}
	\PP\bigpar{V_{1}<\dots<V_{k}\mid V_i=y}.
\end{equation}

For the second probability in \eqref{sw}
we require that $V_1,\dots,V_{i-1}<y$ and
$V_{i+1},\dots,V_k>y$, and furthermore that these two sets of variables are
increasing; since the variables are independent and uniformly distributed,
the probability is, recalling the notation \eqref{g}, 
\begin{equation}
  \frac{y^{i-1}}{(i-1)!} \frac{(1-y)^{k-i}}{(k-i)!} = \frac1{(k-1)!}g_{k,i}(y).
\end{equation}

Similarly, for the first probability in \eqref{sw} we require that the
$\gs(i)$:th smallest of $U_1,\dots,U_k$ is $x$, and that the others come in
the order specified by $\gs$, and the probability of this is
$(k-1)!^{-1}g_{k,\gs(i)}(x)$.
Consequently,
\begin{equation}\label{f1}
  f_{\gs;i}(x,y) = \frac1{(k-1)!^2}\,g_{k,\gs(i)}(x)g_{k,i}(y).
\end{equation}

Furthermore,
\begin{equation}\label{mu2}
	\mu := \E  f_{\gs}(W_1,\dots,W_k) =
	\PP\bigpar{\red(U_{1},\dots,U_{k})=\gs}
	\PP\bigpar{V_{1}<\dots<V_{k}}
	=\frac{1}{k!^2}.
\end{equation}

It follows from \eqref{F}, \eqref{f1}, \eqref{mu2} and \eqref{G} that
$F_\gs(x,y)=G_\gs(y,x)$. Hence \eqref{gSF}--\eqref{gSF2} agree with
\eqref{gS}, and Hoeffding's theorem yields \eqref{t1a}.

Hoeffding's theorem (and its proof sketched above) yields also the
convergence \eqref{t1b} of the covariances. To see that moment convergence
holds also for higher moments, let $m$ be a positive integer.
By \eqref{rep3},
\begin{equation}\label{mom}
  \E\bigpar{\xn\gs-\E\xn\gs}^m
	=\sumx_{j_{11},\dots,j_{k1}} \dotsm \sumx_{j_{1m},\dots,j_{km}}
	\E\prod_{i=1}^m \bigpar{f_\gs(W_{j_{1i}},\dots,W_{j_{ki}})-\mu}
\end{equation}

where the expectation on the right-hand side vanishes unless each index set 
\set{j_{1i},\dots,j_{mi}} contains at least one index shared by another such
set. In this case, however, there are at most $mk-m/2$ distinct indices, and
it follows that the moment \eqref{mom} is a polynomial in $n$ of degree 
at most $mk-m/2$.
In particular, the normalized central moment
$\E\bigpar{(\xn\gs-\E\xn\gs)/n^{k-1/2}}^m=O(1)$.
If $m$ is an even integer, this implies, by standard results on uniform
integrability, that all moments of lower order converge to the corresponding
moments of the limit $Z_\gs$, and the same holds for joint moments. Since $m$ is arbitrary, this shows convergence of all moments.

\end{proof}

\begin{example}\label{E1}

The case $k=1$ is trivial, with $\xn1=n$ deterministic. Indeed,
\eqref{g}--\eqref{G} yield $g_{1,1}(x)=1$ and $G_{1,1}(x,y)=0$.

\end{example}

\begin{example}\label{E2}

The simplest non-trivial example is $k=2$, where $X_{21}(\pi)$ is the 
\emph{number of inversions} in $\pi$. The distribution of this random
variable, for $\pi$ uniformly at random in $\SG_n$, is called the
\emph{Mahonian distribution}, and it is well-known that it is asymptotically
normal, see e.g.\ \cite[Section X.6]{FellerI}.
(See  \cite{CJZ-mahonian} for the case of permutations of multi-sets; it
would be interesting to obtain similar results for other patterns in 
multi-set permutations.)
A simple calculation using \eqref{g}--\eqref{G} yields
\begin{equation}
  G_{21,21}(x,y)=-2\bigpar{x-\tfrac12}\bigpar{y-\tfrac12}
\end{equation}
and \eqref{gS} or \eqref{gS2} yields $\gS_{21,21}=1/36$.
Hence  \refT{T1} in this case yields the well-known
\begin{equation}
  \frac{\xn{21}-\frac12\binom{n}2}{n^{3/2}}
\dto N\bigpar{0,1/36}.
\end{equation}

\end{example}

Example \ref{E1} is the only case when 
the limit $Z_\gs$ in  \refT{T1} vanishes, as we show next.

\begin{theorem}
  \label{T0}
	If $k>1$, then $\gS_{\gs,\gs}>0$ and thus $Z_\gs$ is non-degenerate, 
for every $\gs\in\SG_k$.

\end{theorem}

\begin{proof}

 By \eqref{g},
\begin{equation}
  \sum_{i=1}^k g_{k,i}(x) = 1.
\end{equation}

  Hence \eqref{G} may be written, using Kronecker's delta  $\gd_{i,j}$,
\begin{equation}\label{G2}
  G_\gs(x,y) := \frac{1}{(k-1)!^2} 
	\sum_{i=1}^k \sum_{j=1}^k
	\Bigpar{\gd_{j,\gs(i)}-\frac1k}g_{k,i}(x) g_{k,j}(y).
\end{equation}

For a fixed $k$, the polynomials $g_{k,i}$, $1\le i\le k$, are linearly
independent (and form basis in the $k$-dimensional vector space of polynomials
of degree $\le k-1$).
Hence the $k^2$ tensor products $g_{k,i}(x)g_{k,j}(y)$ are
linearly independent in $L^2(\oi^2)$, and it follows from \eqref{G2} and
\eqref{gS} that if
$k\ge2$, then $G_\gs$ is not identically 0 and thus 
$\gS_{\gs,\gs}=\iint G_\gs (x,y)^2>0$.
\end{proof}

For a given $k$ we have $k!$ patterns $\gs\in\SG_k$ and thus $k!$ limit
variables $Z_\gs$. We have just seen that (if $k>1$)
these are all non-degenerate; however, they are not linearly independent.
For example, the sum $\sum_{\gs\in\SG_k}X_\gs(\pi)=\binom nk$ for every $\pi$,
so the sum is deterministic and it follows that $\sum_{\gs\in \SG_k} Z_\gs = 0$.
Many non-trivial linear combinations vanish too, as is seen by the following
theorem.

\begin{theorem}\label{T2}
  Let $k\ge1$. The $k!$ limit random variables $Z_\gs$, $\gs\in\SG_k$, span
  a linear space of dimension $(k-1)^2$.
\end{theorem}

\begin{proof}

By the definition \eqref{gS}, this linear space, $V$ say, is isomorphic 
(and isometric for the appropriate $L^2$-norms)
to the linear space $V_1$ spanned by the functions $G_\gs$ on $\oi^2$.
Furthermore, by \eqref{G2} and the comments after it, $V_1$ is
isomorphic to the linear space $V_2$ of $k\times k$ matrices spanned by the
matrices 
$A_{\gs} := \bigpar{\gd_{j,\gs(i)}-\frac1k}_{ij=1}^k$.
Let $V_3$ be the space of all $k\times k$ matrices
with all row sums and column sums 0.
Then each matrix $A_{\gs}\in V_3$ and thus $V_2\subseteq V_3$.
Conversely, it is easily seen that each matrix in $V_3$ is a linear
combination of matrices $A_\gs$, for example using the well-known fact that
every doubly stochastic matrix is a convex combination of permutation
matrices.
Hence $V_2=V_3$. Finally, $\dim(V_3)=(k-1)^2$ since a matrix in $V_3$ is
uniquely determined by its upper left corner $(k-1)\times(k-1)$ submatrix
obtained by deleting the last row and column, and conversely this submatrix
may be chosen arbitrarily.
\end{proof}

\begin{example}\label{E3}

There are 6 patterns of length  $k=3$.
Taking them in lexicographic order 123, 132, 213, 231, 312, 321,
and using Maple to calculate the covariance matrix of the 
limit variables $Z_{\gs}$ by \eqref{g}--\eqref{gS}, we find
\begin{equation}\label{cov3}
	\bigpar{\Cov(Z_\gs,Z_\gt)}_{\gs,\gt\in\SG_3}
	=\bigpar{\gS_{\gs,\gt}}_{\gs,\gt\in\SG_3}
	=
	\frac{1}{5!^2}
	\begin {pmatrix} 

26&12&12&-13&-13&-24\\ 

12&14&-1&-6&-6&-13\\ 

12&-1&14&-6&-6&-13\\ 

-13&-6&-6&14&-1&12\\

-13&-6&-6&-1&14&12\\ 

-24&-13&-13&12&12&26

\end{pmatrix}
.
\end{equation}

We note that the asymptotic variances differ between different patterns;
they are  $13/7200$ (for 123 and 321) or $7/7200$ (for the other patterns).

The eigenvalues of the covariance matrix \eqref{cov3} are
\begin{equation}\label{ev3}
	\frac{3}{5!^2}
	\bigpar{25,5,5,1,0,0},
\end{equation}
verifying that this matrix has rank 4 as given by
\refT{T2}. A choice of pairwise orthogonal eigenvectors (in the
corresponding order) is 













\begin{equation}\label{evv3}
	\newcommand\x{\phantom{-}}
  \begin{pmatrix}
		\x2\\ \x1\\ \x1\\ -1\\ -1\\ -2
  \end{pmatrix}
	,\quad
  \begin{pmatrix}
		\x0\\ \x1\\ -1\\ \x0\\ \x0\\ \x0
  \end{pmatrix}
	,\quad
  \begin{pmatrix}
		\x0\\ \x0\\ \x0\\ \x1\\ -1\\ \x0
  \end{pmatrix}
	,\quad
  \begin{pmatrix}
		\x2\\ -1\\ -1\\ -1\\ -1\\ \x2
  \end{pmatrix}
	,\quad
  \begin{pmatrix}
		\x1\\ -1\\ -1\\ \x1\\ \x1\\ -1
  \end{pmatrix}
	,\quad
  \begin{pmatrix}
		1\\1\\1\\1\\1\\1
  \end{pmatrix}
	.
\end{equation}

\end{example}

\begin{remark}

The last eigenvector in \eqref{evv3} corresponds to the trivial fact
mentioned above that the sum of all $Z_\gs$ vanishes.
The fifth eigenvector, also with eigenvalue 0, says that 
\begin{equation}\label{Z=0}
	Z_{123}+Z_{231}+Z_{312}-Z_{132}-Z_{213}-Z_{321}  =0.
\end{equation}

Let $Y(\pi)$ be the corresponding number
\begin{equation}\label{Y}
	Y(\pi):=
	X_{123}(\pi)+X_{231}(\pi)+X_{312}(\pi)-X_{132}(\pi)-X_{213}(\pi)-X_{321}(\pi),
\end{equation}
and let $Y_n:=Y(\pi)$ with $\pi$ chosen uniformly at random in $\SG_n$.
(Note that $Y(\pi)$ is the sum of the signs of the $\binom n3$ permutations
$\red\bigpar{\pi_{i_1}\pi_{i_2}\pi_{i_3}}$.)
\refT{T1} and \eqref{Z=0} thus say that, as $n\to\infty$,
$n^{-5/2}Y_n\dto 0$. However, in this case, the random variable $Y_n$ does
not vanish identically. (Take $\pi$ as the identity permutation.)
Using the same methods as in Section~\ref{SECmomsSn}, we can show that
\begin{equation}
	\Var(Y_n) = \frac{n^{2} (n-1) (n-2)}{18}.
\end{equation}

In particular, we have that the leading term of $\Var(Y_n)$ is $\frac1{18}n^4$, 
i.e. of order $n^{2k-2}$ instead of
$n^{2k-1}$ as in the cases when \refT{T1} yields a non-degenerate limit.
In such cases, one can use a more advanced version of Hoeffding's argument
above and show that there is an asymptotic distribution that can be
represented as an (infinite) polynomial of degree 2 in normal random
variables; this polynomial can further be diagonalized as a linear
combination of squares of independent normal variables, see e.g.{}
\cite{Rubin-Vitale} and \cite[Section 11.1]{SJIII}. 
In the present case this leads to
\begin{equation}\label{Y*}
  n^{-2} Y_n \dto  Y^* = \sum_{\substack{\ell,m=-\infty\\\ell,m\neq0}}^\infty
\frac1{2\pi^2 \ell m}\bigpar{\xi_{\ell,m}^2-1},
\end{equation}
where $\xi_{\ell,m}$ are i.i.d.\ standard normal random variables.
(We omit the details but note that the bilinear form in 
\cite[Corollary 11.5(iii)]{SJIII} 
in this case after some calculation
turns out to correspond to
the convolution operator on $L^2(\mathbb T^2)$ given by convolution with
$H(x,y)=\frac16(2x-1)(2y-1)$ (where we identify the group $\mathbb T$ with
$[0,1)$);  
 hence its eigenvalues are the Fourier coefficients 
$\widehat H(\ell,m) = -1/(6\pi^2\ell m)$, which yields the coefficients in
 \eqref{Y*}.)
Note that, since $\Var(\xi_{\ell,m}^2)=2$, 
\begin{equation}\label{Y*v}
	\Var Y^* = \sum_{\substack{\ell,m=-\infty\\\ell,m\neq0}}^\infty
	\frac2{4\pi^4 \ell^2 m^2} =\frac1{18},
\end{equation}
 in accordance with the asymptotic formula $\Var(Y_n)\sim n^4/18$.
Furthermore, the representation \eqref{Y*} of the limit $Y$ yields its
moment generating function as
\begin{equation}
	\E e^{tY^*} 
	= \prod_{\substack{\ell,m=-\infty\\\ell,m\neq0}}^\infty
	\Bigpar{1-\frac{2t}{2\pi^2 \ell m}}^{-1/2}
	= \prod_{\ell,m=1}^\infty
	\Bigpar{1-\frac{t^2}{\pi^4 \ell^2 m^2}}^{-1}
	= \prod_{m=1}^\infty \frac{t/m\pi}{\sin(t/m\pi)},
	\qquad |\Re t|<\pi^2.
\end{equation}

This type of limit is typical of the degenerate cases that can occur for
certain linear combinations of pattern counts. 
It is also possible to obtain higher degeneracies in special cases, with
variance of still lower order and a limit that is a polynomial of higher
degree in infinitely many normal variables; one example is to generalize 
\eqref{Y} by taking, for any fixed $k\ge3$, the sum of the signs of the
$\binom nk$ patterns of length $k$ occurring in $\pi$. It can be seen that
for this example, $\Var(Y_n)$ is a polynomial in $n$ of degree $k+1$ only
(instead of the typical $2k-1$),
because  all higher order terms cancel in this highly symmetric example.
\end{remark}

\begin{example}

There are 24 patterns of length  $k=4$.  
A calculation as in Example \ref{E3} of the covariance matrix yields a
$24\times24$ matrix 
of rank $(4-1)^2=9$. The 9 non-zero eigenvalues are
\begin{equation}\label{ev4}
	\frac{8}{7!^2}
	\bigpar{441,147,147,49,21,21,7,7,1}.
\end{equation}

Similarly, for $k=5$ the covariance matrix is a $120\times120$ matrix with 
the $4^2=16$ non-zero eigenvalues
\begin{equation}\label{ev5}
	\frac{30}{9!^2}
	\bigpar{7056,3024,3024,1296,756,756,324,324,84,84,81,36,36,9,9,1}.
\end{equation}

The fact that the eigenvalues in \eqref{ev3}, \eqref{ev4} and \eqref{ev5}
all are simple rational numbers suggests that there is a general structure
(valid for all $k$) 
for these eigenvalues, and presumably also for the corresponding
eigenvectors;
it would be interesting to know more about this.

\end{example}


\section{Conclusion}\label{SECconcl}

In this article, we studied the moments and mixed moments of the random variables $X_{\sigma}(\pi)$ for a number of patterns $\sigma$, where $\pi$ may be chosen from $\SG_{n}$ or a pattern avoiding set $\SG_{n}(\tau)$. In addition, we prove that for any two patterns, the corresponding random variables are joint asymptotically normal when the permutations are drawn from $\SG_{n}$. The contrasting computational approach can compute a number of moments and mixed moments as well as derive (rigorous) formulas for the lower moments. We anticipate that this approach could be extended to provide an alternative proof to the joint asymptotic normality of multiple random variables, but we leave this as ``future work''.

In the setting where the permutations are chosen from the pattern avoiding set $\SG_{n}(\tau)$ (for some fixed pattern $\tau$), much less is known. Others have recently studied the total number of occurrences of a pattern in these sets, which is equivalent to the expected value (i.e., the first moment) of the random variable $X_{\sigma}$, generally for when both $\sigma, \tau \in \SG_{3}$. Our approach allows us to quickly compute many empirical moments, far beyond the first moment. We expect that a more thorough analysis of these higher moments will uncover interesting properties and that in some cases, these higher moments will also have closed form formulas. In addition, the random variables for some patterns appear to \emph{not} be asymptotically normal (whereas in the case where permutations are drawn from $\SG_{n}$, they are asymptotically normal for every pattern \cite{Bona}). It would be interesting to understand which patterns (if any) have corresponding random variables that are asymptotically normal when permutations are drawn from $\SG_{n}(\tau)$.


\newcommand\arxiv[1]{\texttt{arXiv:#1.}}

\newcommand\arXiv{\arxiv}

\def\nobibitem#1\par{}

\end{document}